\documentclass[a4paper,11pt]{amsart}
\usepackage{latexsym}
\usepackage{amsmath,amssymb}
\usepackage{amsmath}

\usepackage[all]{xy}
\usepackage{enumerate}

\newtheorem{theo}{Theorem}[section]
\newtheorem{coll}[theo]{Corollary}
\newtheorem{lemm}[theo]{Lemma}
\newtheorem{prop}[theo]{Proposition}
\newtheorem{defn}[theo]{Definition}
\newtheorem{ex}[theo]{Example}
\newtheorem{rem}[theo]{Remark}

\newcommand{\Hom}{{\rm Hom}}

\begin{document}
\sloppy

\title[Strongly regular and strongly Baer objects]{Strongly Rickart objects in abelian categories. \\
Applications to strongly regular \\ and strongly Baer objects}

\author[S.Crivei]{Septimiu Crivei}

\address{Faculty of Mathematics and Computer Science, Babe\c s-Bolyai University, Str. M. Kog\u alniceanu 1,
400084 Cluj-Napoca, Romania} \email{crivei@math.ubbcluj.ro}

\author[G. Olteanu]{Gabriela Olteanu}

\address{Department of Statistics-Forecasts-Mathematics, Babe\c s-Bolyai University, Str. T. Mihali 58-60, 400591
Cluj-Napoca, Romania} \email{gabriela.olteanu@econ.ubbcluj.ro}


\begin{abstract} We show how the theory of (dual) strongly relative Rickart objects may be employed in order to 
study strongly relative regular objects and (dual) strongly relative Baer objects in abelian categories. 
For each of them, we prove general properties, we analyze the behaviour with respect to (co)products, 
and we study the transfer via functors. We also give applications to Grothendieck categories, 
(graded) module categories and comodule categories.
\end{abstract}

\date{February 5, 2018}

\maketitle

\section{Introduction}

The theory of (dual) strongly relative Rickart objects developed in the companion paper \cite{CO1} is systematically 
used in the present paper in order to study strongly relative regular objects and 
(dual) strongly relative Baer objects in abelian categories. We also give applications to Grothendieck categories, 
(graded) module categories and comodule categories.
For an introduction and motivation of the topic as well as for all needed concepts and notation the reader is referred to \cite{CO1}. 
Usually the statements of our results have two parts, out of which we only prove the first one, the second one following by
the duality principle in abelian categories.

In Section 2 we define strongly relative regular objects in abelian categories. 
Let $M$ and $N$ be objects of an abelian category $\mathcal{A}$. Then $N$ is called 
\emph{strongly $M$-regular} if $N$ is strongly $M$-Rickart and dual strongly $M$-Rickart. 
Also, $N$ is called \emph{strongly self-regular} if $N$ is strongly $N$-regular. We show that $M$ 
is strongly self-regular if and only if ${\rm End}_{\mathcal{A}}(M)$ is a strongly regular ring 
if and only if $M$ is self-regular and weak duo. Also, we prove that $N$ is strongly $M$-regular if and only if 
$N$ is strongly $M$-Rickart and $M$ is direct $N$-injective if and only if $N$ is dual strongly $M$-Rickart and 
direct $M$-projective.

In Section 3 we study (co)products of strongly relative regular objects. We prove that if 
$M$, $N_1,\dots,N_n$ are objects of an abelian category $\mathcal{A}$, 
then $\bigoplus_{i=1}^n N_i$ is strongly $M$-regular
if and only if $N_i$ is strongly $M$-regular for every $i\in \{1,\dots,n\}$.
We show that if $M=\bigoplus_{i\in I}M_i$ is a direct sum decomposition of 
an object $M$ of an abelian category $\mathcal{A}$, then $M$ is strongly self-regular if and only if 
$M_i$ is strongly self-regular for every $i\in I$, and ${\rm Hom}_{\mathcal{A}}(M_i,M_j)=0$ for every $i,j\in I$ with $i\neq j$.
We derive a corollary on the structure of strongly self-regular modules over a Dedekind domain. 

In Section 4 we deal with the transfer of the strong relative regular property via functors. 
We show various results involving fully faithful functors, adjoint pairs and adjoint triples of functors. 
Let $(L,R)$ be an adjoint pair of covariant functors $L:\mathcal{A}\to \mathcal{B}$ and
$R:\mathcal{B}\to \mathcal{A}$ between abelian categories such that $L$ is exact, 
and let $M$ and $N$ be objects of $\mathcal{B}$ such that $M,N\in {\rm Stat}(R)$. 
Then we prove that the following are equivalent: $(i)$ $N$ is strongly $M$-regular in $\mathcal{B}$; 
$(ii)$ $R(N)$ is strongly $R(M)$-regular in $\mathcal{A}$ and for every morphism $f:M\to N$, ${\rm Ker}(f)$ is $M$-cyclic; 
$(iii)$ $R(N)$ is strongly $R(M)$-regular in $\mathcal{A}$ and for every morphism $f:M\to N$, ${\rm Ker}(f)\in {\rm Stat}(R)$.

In Section 5 we define (dual) strongly relative Baer objects in abelian categories. 
Let $M$ and $N$ be objects of an abelian category $\mathcal{A}$. Then $N$ is called 
\emph{strongly $M$-Baer} if for every family $(f_i)_{i\in I}$ with each $f_i\in \Hom_{\mathcal{A}}(M,N)$, $\bigcap_{i\in I}
{\rm Ker}(f_i)$ is a fully invariant direct summand of $M$. Also, $N$ is called \emph{strongly self-Baer} if 
$N$ is strongly $N$-Baer. We show that $M$ is strongly self-Baer if and only if $M$ is self-Baer and weak duo. 
Also, if there exists the product $M^I$ for every set $I$, then $M$ is strongly self-Baer 
if and only if $M$ is strongly self-Rickart and has the strong summand intersection property.
If there exists the product $N^I$ for every set $I$, then we prove that 
$N$ is strongly $M$-Baer and $M$-$\mathcal{K}$-cononsingular if and only if 
$M$ is strongly extending and $N$ is $M$-$\mathcal{K}$-nonsingular.

In Section 6 we study (co)products of (dual) strongly relative Baer objects. We prove that if 
$M$, $N_1,\dots,N_n$ are objects of an abelian category $\mathcal{A}$, 
then $\bigoplus_{i=1}^n N_i$ is strongly $M$-Baer
if and only if $N_i$ is strongly $M$-Baer for every $i\in \{1,\dots,n\}$.
We show that if $M=\bigoplus_{i\in I}M_i$ is a direct sum decomposition of 
an object $M$ of an abelian category $\mathcal{A}$, then $M$ is strongly self-Baer if and only if 
$M_i$ is strongly self-Baer for every $i\in I$, ${\rm Hom}_{\mathcal{A}}(M_i,M_j)=0$ for every $i,j\in I$ with $i\neq j$,
and $N=\bigoplus_{i\in I}(N\cap M_i)$ for every direct summand $N$ of $M$.
We derive a corollary on the structure of strongly self-Baer modules over a Dedekind domain. 

In Section 7 we study the transfer of the strong relative Baer property via functors, similarly to Section 4. 
For a right $R$-module $M$ with $S={\rm End}_R(M)$, we show that the following are equivalent:
$(i)$ $M$ is a strongly self-Baer right $R$-module; 
$(ii)$ $S$ is a strongly self-Baer right $S$-module and for every set $I$ and for every family $(f_i)_{i\in I}$ with each $f_i\in
S$, $\bigcap_{i\in I}{\rm Ker}(f_i)$ is $M$-cyclic;
$(iii)$ $S$ is a strongly self-Baer right $S$-module and for every set $I$ and for every family $(f_i)_{i\in I}$ with each $f_i\in
S$, $\bigcap_{i\in I}{\rm Ker}(f_i)\in {\rm Stat}({\rm Hom}_R(M,-))$;
$(iv)$ $S$ is a strongly self-Baer right $S$-module and for every set $I$ and for every family $(f_i)_{i\in I}$ with each $f_i\in
S$, $\bigcap_{i\in I}{\rm Ker}(f_i)$ is a locally split submodule;
$(v)$ $S$ is a strongly self-Baer right $S$-module and $M$ is quasi-retractable.

\section{Strongly relative regular objects}

In this section we begin to systematically apply our theory of (dual) strongly relative Rickart objects 
to the study of some corresponding regular-type objects of an abelian category, called strongly relative regular objects.

Let us first recall the concept of relative regular object in a category.

\begin{defn} \cite[Definition~2.1]{DNTD} \rm Let $M$ and $N$ be objects of an arbitrary category $\mathcal{C}$. Then $N$ is called:
\begin{enumerate}
\item \emph{$M$-regular} if every morphism $f:M\to N$ in $\mathcal{C}$ has a generalized inverse, 
in the sense that there exists a morphism $g:N\to M$ in $\mathcal{C}$ such that $fgf=f$.
\item \emph{self-regular} if $N$ is $N$-regular.
\end{enumerate}
\end{defn}

Relative regular objects of abelian categories are characterized as follows.

\begin{theo} \cite[Proposition~3.1]{DNTD}, \cite[Corollary~2.3]{CK}
Let $M$ and $N$ be objects of an abelian category $\mathcal{A}$. Then 
$N$ is $M$-regular if and only if for every morphism $f:M\to N$, ${\rm Ker}(f)$ is a direct summand of $M$ 
and ${\rm Im}(f)$ is a direct summand of $N$, that is, $N$ is $M$-Rickart and dual $M$-Rickart.
\end{theo}

In a similar way, we introduce the following concept.  

\begin{defn} \label{d:strreg} \rm Let $M$ and $N$ be objects of an abelian category $\mathcal{A}$. Then $N$ is called:
\begin{enumerate}
\item \emph{strongly $M$-regular} if for every morphism $f:M\to N$, ${\rm Ker}(f)$ is a fully invariant direct summand of $M$ and 
${\rm Im}(f)$ is a fully invariant direct summand of $N$, that is, $N$ is strongly $M$-Rickart and dual strongly $M$-Rickart.
\item \emph{strongly self-regular} if $N$ is strongly $N$-regular.
\end{enumerate}
\end{defn}

As relative regularity has its root in von Neumann regularity of rings, soon we shall 
see that strong relative regularity is related to strong regularity of rings. 
Recall that a ring $R$ is called \emph{strongly regular} if 
for every $a\in R$ there exists an element $b\in R$ such that $a=a^2b$ (equivalently, 
for every $a\in R$ there exists an element $b\in R$ such that $a=ba^2$) \cite{AK}.
We also recall the following well known characterization of strongly regular rings.

\begin{prop} \label{p:strreg} A ring $R$ is strongly regular if and only if $R$ is von Neumann regular and abelian.
\end{prop}

\begin{prop} \label{p:Mstrreg} Let $M$ be an object of an abelian category $\mathcal{A}$. 
Then $M$ is strongly self-regular if and only if its endomorphism ring ${\rm End}_{\mathcal{A}}(M)$ is strongly regular.
\end{prop}

\begin{proof} Assume that $M$ is strongly self-regular. Then $M$ is self-regular, 
and so ${\rm End}_{\mathcal{A}}(M)$ is von Neumann regular. 
Let $e\in {\rm End}_{\mathcal{A}}(M)$ be an idempotent and $h\in {\rm End}_{\mathcal{A}}(M)$. Since every idempotent splits, 
there exists an object $K$ and morphisms $k:K\to M$ and $p:M\to K$ such that $kp=e$ and $pk=1_K$. 
Since $k$ is a kernel and $M$ is strongly self-Rickart, $hek=k\alpha$ for some morphism $\alpha:K\to K$. It follows that 
$ehek=ek\alpha=kpk\alpha=k\alpha=hek$, hence $ehe=ehekp=hekp=he$. Thus, $e$ is left semicentral. 
Hence ${\rm End}_{\mathcal{A}}(M)$ is abelian. Then ${\rm End}_{\mathcal{A}}(M)$ is strongly regular by Proposition \ref{p:strreg}.

Conversely, assume that ${\rm End}_{\mathcal{A}}(M)$ is strongly regular. 
Then ${\rm End}_{\mathcal{A}}(M)$ is von Neumann regular by Proposition \ref{p:strreg},
and so $M$ is self-regular. It follows that $M$ is self-Rickart and dual self-Rickart. 
We claim that $M$ is weak duo. To this end, let $k:K\to M$ be a section and $p:M\to K$ the canonical projection. 
Then $pk=1_K$ and $e=kp\in {\rm End}_{\mathcal{A}}(M)$ is idempotent. By Proposition \ref{p:strreg}, 
${\rm End}_{\mathcal{A}}(M)$ is abelian, and so $e$ is central. 
It follows that $hkp=he=eh=kph$, hence $hk=kphk$. Thus, $k$ is a fully invariant section, and so $M$ is weak duo. 
Finally, $M$ is strongly self-Rickart and dual strongly self-Rickart by \cite[Corollary~2.10]{CO1}. 
Hence $M$ is strongly self-regular.
\end{proof}

\begin{coll} \label{c:wduo-reg} Let $M$ be an object of an abelian category $\mathcal{A}$. Then the following are equivalent:
\begin{enumerate}[(i)]
\item $M$ is strongly self-regular.
\item $M$ is self-regular and weak duo. 
\item $M$ is self-regular and ${\rm End}_{\mathcal{A}}(M)$ is abelian.
\item For every endomorphism $f:M\to M$, $M={\rm Ker}(f)\oplus {\rm Im}(f)$.  
\end{enumerate}
\end{coll}

\begin{proof} The equivalence of the first three conditions follows by \cite[Corollary~2.10]{CO1} and \cite[Proposition~2.14]{CO1}.
The equivalence $(i)\Leftrightarrow (iv)$ follows in the same way as \cite[Theorem~2.2]{LRR13}, 
whose proof works in abelian categories. 
\end{proof}

\begin{ex} \label{e:strreg} \rm (a) Consider the ring $A=\mathbb{Z}_2^{\mathbb{N}}$, and its subrings 
$T=\{(a_n)_{n\in \mathbb{N}}\mid a_n \textrm{ is eventually constant}\}$ and
$I=\{(a_n)_{n\in \mathbb{N}}\mid a_n=0 \textrm{ eventually}\}=\mathbb{Z}_2^{(\mathbb{N})}$. 
Let $R=\begin{pmatrix}T&T/I\\0&T/I\end{pmatrix}$, the idempotent $e=\begin{pmatrix}(1,1,\dots)&I\\0&I\end{pmatrix}\in R$ 
and $M=eR=\begin{pmatrix}T&T/I\\0&0\end{pmatrix}$. 
Then $M$ is a self-Rickart right $R$-module \cite[Example~2.18]{LRR11} 
and a dual self-Rickart right $R$-module \cite[Example~4.1]{LRR10}, 
hence $M$ is self-regular. But ${\rm End}_R(M)=\begin{pmatrix}T&0\\0&0\end{pmatrix}$ is commutative, 
hence $M$ is strongly self-regular by Corollary \ref{c:wduo-reg}.

(b) The full $2\times 2$ matrix ring $R=M_2(K)$ over a field $K$ is a self-regular right $R$-module, 
which is not strongly self-regular.
\end{ex}

\begin{theo} \label{t:epimono-reg} Let $r:M\to M'$ be an epimorphism and $s:N'\to N$ a monomorphism in an abelian category
$\mathcal{A}$. If $N$ is strongly $M$-regular, then $N'$ is strongly $M'$-regular.
\end{theo}

\begin{proof} This follows by \cite[Theorem~2.17]{CO1}. 
\end{proof}

\begin{coll} \label{c:summand-reg} Let $M$ and $N$ be objects of an abelian category $\mathcal{A}$, $M'$ a direct summand 
of $M$ and $N'$ a direct summand of $N$. If $N$ is strongly $M$-regular, then $N'$ is strongly $M'$-regular.
\end{coll}

\begin{proof} This follows by \cite[Corollary~2.18]{CO1}. 
\end{proof}

Next we explore further relationships between our concepts and strong relative regularity. 

Let $M$ and $N$ be objects of an abelian category $\mathcal{A}$. We recall some generalizations of injectivity and
projectivity that are useful in the study of relative regular objects. Following \cite[p.~220]{NZ}, $M$ is called
\emph{direct $N$-injective} if every subobject of $N$ isomorphic to a direct summand of $M$ is a direct summand of $N$.
Dually, $N$ is called \emph{direct $M$-projective} if for every factor object of $M/K$ isomorphic to a direct summand of $N$,
$K$ is a direct summand of $M$. For $M=N$ the above notions particularize to direct injectivity and direct projectivity
respectively. In this case, a direct injective object is also called a \emph{$C_2$-object}, while a direct projective
object is also called a \emph{$D_2$-object}. 


\begin{theo} \label{t:char} Let $M$ and $N$ be objects of an abelian category $\mathcal{A}$. Then the following are
equivalent:
\begin{enumerate}[(i)]
\item $N$ is strongly $M$-regular.
\item $N$ is strongly $M$-Rickart and $M$ is direct $N$-injective.
\item $N$ is dual strongly $M$-Rickart and direct $M$-projective.
\end{enumerate}
\end{theo}

\begin{proof} (i)$\Rightarrow$(ii) Assume that $N$ is strongly $M$-regular. Then $N$ is strongly $M$-Rickart by 
Proposition~\ref{p:Mstrreg} and $N$ is $M$-regular. It follows that $M$ is direct $N$-injective by \cite[Theorem~5.3]{CK}. 

(ii)$\Rightarrow$(i) Assume that $N$ is strongly $M$-Rickart and $M$ is direct $N$-injective. 
Then $N$ is $M$-Rickart and $M$ is weak duo by \cite[Proposition~2.9]{CO1}. By \cite[Theorem~5.3]{CK}, $N$ is dual $M$-Rickart. 
Then $N$ is dual strongly $M$-Rickart by \cite[Proposition~2.9]{CO1}. Hence $N$ is strongly $M$-regular.

The equivalence (i)$\Leftrightarrow$(iii) follows by duality.
\end{proof}

\begin{coll} Let $M$ be an object of an abelian category $\mathcal{A}$. Then the following are equivalent:
\begin{enumerate}[(i)]
\item $M$ is strongly self-regular.
\item $M$ is strongly self-Rickart and direct injective.
\item $M$ is dual strongly self-Rickart and direct projective.
\end{enumerate}
\end{coll}

Theorem~\ref{t:char} allows one to use properties of strongly relative Rickart objects and direct relative injectivity (or
equivalently, properties of dual strongly relative Rickart objects and direct relative projectivity) in order to deduce
properties of strongly relative regular objects. We shall show several results which underline this technique.  

\begin{theo} \label{t:extensions-reg} Let $\mathcal{A}$ be an abelian category. 
\begin{enumerate} \item Consider a short exact sequence $$0\to N_1\to N\to N_2\to 0$$ and an object $M$ of $\mathcal{A}$
such that $N_1$ and $N_2$ are strongly $M$-regular. Then $N$ is strongly $M$-regular.  
\item Consider a short exact sequence $$0\to M_1\to M\to M_2\to 0$$ and an object $N$ of $\mathcal{A}$ such that $N$ is
dual strongly $M_1$-regular and dual strongly $M_2$-regular. Then $N$ is dual strongly $M$-regular. 
\end{enumerate}
\end{theo}

\begin{proof} This follows by \cite[Theorem~2.19]{CO1} and \cite[Lemma~5.2]{CK}.
\end{proof}

Next we give some applications to graded rings and modules.

\begin{defn} \rm A $G$-graded ring $R=\bigoplus_{\sigma\in G}R_{\sigma}$ is called \emph{strongly gr-regular} 
if for every $x_{\sigma}\in R_{\sigma}$ there exists $y\in R$ (which can be assumed to be in $R_{\sigma^{-1}}$) 
such that $x_{\sigma}=x^2_{\sigma}y$.  
\end{defn}

Now we may easily give an analogue of \cite[Theorem~5.2]{DNTD} for strongly gr-regular rings.

\begin{theo} \label{t:grring} Let $R=\bigoplus_{\sigma\in G}R_{\sigma}$ be a $G$-graded ring. 
Then $R$ is strongly gr-regular if and only if 
$R(\sigma)$ is a strongly $R$-regular graded right $R$-module for every $\sigma\in G$.
\end{theo}

\begin{proof} Assume first that $R$ is strongly gr-regular. Let $\sigma\in G$ and let $f:R\to R(\sigma)$ be a homomorphism 
of graded right $R$-modules. For $x_{\sigma}=f(1)$ there exists $y\in R_{\sigma^{-1}}$ such that 
$x_{\sigma}=x^2_{\sigma}y$. Then $g:R(\sigma)\to R$ defined by $g(r)=yr$ is a homomorphism of graded right $R$-modules,
and $f=f^2g$. Hence $R(\sigma)$ is a strongly $R$-regular graded right $R$-module.

Conversely, assume that $R(\sigma)$ is a strongly $R$-regular graded right $R$-module for every $\sigma\in G$. 
Let $\sigma\in G$ and $x_{\sigma}\in R_{\sigma}$. Consider the homomorphism $f:R\to R(\sigma)$ 
of graded right $R$-modules defined by $f(r)=x_{\sigma}r$. Then there exists a homomorphism $g:R(\sigma)\to R$ 
of graded right $R$-modules such that $f=f^2g$. Then $x_{\sigma}=x^2_{\sigma}f(1)$, which shows that $R$ is strongly gr-regular. 
\end{proof}

\begin{coll} Let $R=\bigoplus_{\sigma\in G}R_{\sigma}$ be a $G$-graded ring. 
\begin{enumerate}[(i)]
\item If $R$ is strongly gr-regular, then $R_{\sigma}$ is strongly $R_e$-regular for every $\sigma\in G$. 
In particular, $R$ is strongly $R_e$-regular.
\item If $R$ is strongly graded and $R_{\sigma}$ is strongly $R_e$-regular for every $\sigma\in G$, 
then $R$ is strongly gr-regular.
\end{enumerate}
\end{coll}

\begin{proof} (i) Let $\sigma\in G$ and let $f:R_e\to R_{\sigma}$ be a homomorphism of right $R_e$-modules.
Let $x_{\sigma}=f(1)$. There exists $y\in R_{\sigma^{-1}}$ such that $x_{\sigma}=x^2_{\sigma}y$.
Then $g:R_{\sigma}\to R_e$ defined by $g(r)=yr$ is a homomorphism of right $R_e$-modules, 
and $f=f^2g$. Hence $R_{\sigma}$ is strongly $R_e$-regular.

(ii) This follows by Theorem \ref{t:grring}. 
\end{proof}

\section{(Co)products of strongly relative regular objects}

We show several results on (co)products of (dual) strongly relative regular objects in abelian categories. 
They are naturally obtained by using the two main ways of deducing results of strongly relative regular objects from 
the theory of (dual) strongly relative Rickart objects, namely Definition \ref{d:strreg} and Theorem \ref{t:char}.

\begin{theo} \label{t:dsreg} Let $\mathcal{A}$ be an abelian category.
\begin{enumerate} 
\item Let $M$ and $N_1,\dots,N_n$ be objects of $\mathcal{A}$. Then $\bigoplus_{i=1}^n N_i$ is strongly $M$-regular
if and only if $N_i$ is strongly $M$-regular for every $i\in \{1,\dots,n\}$.
\item Let $M_1,\dots,M_n$ and $N$ be objects of $\mathcal{A}$. Then $N$ is strongly $\bigoplus_{i=1}^n M_i$-regular
if and only if $N$ is strongly $M_i$-regular for every $i\in \{1,\dots,n\}$.
\end{enumerate}
\end{theo}

\begin{proof} (1) By \cite[Theorem~3.1]{CO1}, Theorem \ref{t:char} and \cite[Lemma~5.2]{CK}, 
$\bigoplus_{i=1}^n N_i$ is strongly $M$-regular 
if and only if [$\bigoplus_{i=1}^n N_i$ is strongly $M$-Rickart, and $M$ is direct $\bigoplus_{i=1}^n N_i$-injective] 
if and only if [$N_i$ is strongly $M$-Rickart, and $M$ is direct $N_i$-injective for every $i\in \{1,\dots,n\}$] 
if and only if $N_i$ is strongly $M$-regular for every $i\in \{1,\dots,n\}$.
\end{proof}

\begin{rem} \rm \cite[Example~3.3 (a)]{CO1} also shows that Theorem~\ref{t:dsreg} does not hold in general for arbitrary coproducts. 
\end{rem}

As an application of Theorem~\ref{t:dsreg}, we give the following property involving coproducts of (dual) strongly self-Rickart
objects.

\begin{theo} \label{t:59} Let $M_1,\dots,M_n$ be objects of an abelian category $\mathcal{A}$.
\begin{enumerate}
\item Assume that $M_i$ is direct $M_j$-injective for every $i,j\in \{1,\dots,n\}$. Then $\bigoplus_{i=1}^n M_i$ is
strongly self-Rickart if and only if $M_i$ is strongly $M_j$-Rickart for every $i,j\in \{1,\dots,n\}$. 
\item Assume that $M_i$ is direct $M_j$-projective for every $i,j\in \{1,\dots,n\}$. Then $\bigoplus_{i=1}^n M_i$ is
dual strongly self-Rickart if and only if $M_i$ is dual strongly $M_j$-Rickart for every $i,j\in \{1,\dots,n\}$. 
\end{enumerate}
\end{theo}

\begin{proof} (1) If $\bigoplus_{i=1}^n M_i$ is strongly self-Rickart, then $M_i$ is strongly $M_j$-Rickart for every $i,j\in
\{1,\dots,n\}$ by \cite[Corollary~2.18]{CO1}.

Conversely, assume that $M_i$ is strongly $M_j$-Rickart for every $i,j\in \{1,\dots,n\}$. Since $M_i$ is direct $M_j$-injective
for every $i,j\in \{1,\dots,n\}$, it follows by Theorem~\ref{t:char} that $M_i$ is strongly $M_j$-regular for every
$i,j\in \{1,\dots,n\}$. Then $\bigoplus_{i=1}^n M_i$ is strongly self-regular by Theorem~\ref{t:dsreg}. Finally,
$\bigoplus_{i=1}^n M_i$ is strongly self-Rickart by Theorem~\ref{t:char}.
\end{proof}

We also have the following result on arbitrary (co)products under some finiteness conditions.

\begin{coll} \label{c:fg-reg} Let $\mathcal{A}$ be an abelian category.
\begin{enumerate} \item Assume that $\mathcal{A}$ has coproducts, let $M$ be a finitely generated object of
$\mathcal{A}$, and let $(N_i)_{i\in I}$ be a family of objects of $\mathcal{A}$. Then $\bigoplus_{i\in I}
N_i$ is strongly $M$-regular if and only if $N_i$ is strongly $M$-regular for every $i\in I$.

\item Assume that $\mathcal{A}$ has coproducts, let $N$ be a finitely cogenerated object of $\mathcal{A}$, and let
$(M_i)_{i\in I}$ be a family of objects of $\mathcal{A}$ such that $N$ is strongly $M_i$-regular for every $i\in I$. Then
$N$ is strongly $\prod_{i\in I} M_i$-regular if and only if $N$ is strongly $M_i$-regular for every $i\in I$.
\end{enumerate}
\end{coll}

\begin{proof} (1) By \cite[Corollary~3.2]{CO1}, Theorem~\ref{t:char} and the immediate analogue of 
\cite[16.2]{Wis} for abelian categories, $\bigoplus_{i\in I} N_i$ is strongly $M$-regular 
if and only if [$\bigoplus_{i\in I} N_i$ is strongly $M$-Rickart, and $M$ is direct $\bigoplus_{i\in I} N_i$-injective] 
if and only if [$N_i$ is strongly $M$-Rickart, and $M$ is direct $N_i$-injective for every $i\in \{1,\dots,n\}$] 
if and only if $N_i$ is strongly $M$-regular for every $i\in \{1,\dots,n\}$.
\end{proof}

Let $G$ be a finite group and let $R=\bigoplus_{\sigma\in G}R_{\sigma}$ be a $G$-graded ring. 
Then one may associate to $R$ a ring $R\# G^*$, called \emph{smash product}, defined as follows \cite[Chapter 7]{Nasta-04}:
$R\# G^*=\bigoplus_{x\in G} Rp_x$, where the family $(p_x)_{x\in G}$ is a basis of $R\# G^*$, and 
the multiplication is defined by $(r p_x)(s p_y)=rs_{xy^{-1}}p_y$ for every $r,s\in R$ and $x,y\in G$. 

Now we can give an analogue of \cite[Corollary~5.4]{DNTD} for strongly gr-regular rings. 

\begin{coll} Let $G$ be a finite group and let $R=\bigoplus_{\sigma\in G}R_{\sigma}$ be a $G$-graded ring. 
Then $R$ is strongly gr-regular if and only if the smash product $R\#G^*$ is a strongly regular ring. 
\end{coll}

\begin{proof} Assume first that $R$ is strongly gr-regular. Denote $U=\bigoplus_{\sigma}R(\sigma)$. 
Since $R$ is strongly gr-regular, $R(\sigma)$ is a strongly $R$-regular graded right $R$-module 
for every $\sigma\in G$ by Theorem \ref{t:grring}. But $R$ is a finitely generated graded right $R$-module, 
hence $U$ is a strongly $R$-regular graded right $R$-module by Corollary \ref{c:fg-reg}. 
Consider the $\sigma$-suspension functor $T_{\sigma}:{\rm gr}(R)\to {\rm gr}(R)$ defined by 
$T_{\sigma}(M)=M_{\sigma}$ for every graded right $R$-module $M$. Since $T_{\sigma}$ is an isomorphism of categories 
for every $\sigma\in G$, $U=U(\sigma)$ is a strongly $R(\sigma)$-regular graded right $R$-module.
Since $G$ is finite, it follows that $U$ is a strongly self-regular graded right $R$-module by Theorem \ref{t:dsreg}. 
Then $R\#G^*\cong {\rm End}_{{\rm gr}(R)}(U)$ \cite[Theorem~7.2.1]{Nasta-04} is a strongly regular ring by Proposition \ref{p:Mstrreg}. 

Conversely, assume that $R\#G^*$ is a strongly regular ring. Then ${\rm End}_{{\rm gr}(R)}(U)\cong R\#G^*$ is 
a strongly regular ring, hence $U$ is a strongly self-regular graded right $R$-module by Proposition \ref{p:Mstrreg}. 
Using again the $\sigma$-suspension functor $T_{\sigma}$, it follows that $U$ is a strongly $R$-regular graded right $R$-module.
Then $R(\sigma)$ is a strongly $R$-regular graded right $R$-module for every $\sigma\in G$ by 
Corollary \ref{c:summand-reg}. Finally, $R$ is strongly gr-regular by Theorem \ref{t:grring}. 
\end{proof}

The next result gives a necessary condition for an infinite (co)product of objects to be a strongly self-regular object.

\begin{prop} \label{p:relreg} Let $(M_i)_{i\in I}$ be a family of objects of an abelian category $\mathcal{A}$
such that $\prod_{i\in I} M_i$ or $\bigoplus_{i\in I} M_i$ is a strongly self-regular object. 
Then $M_i$ is strongly $M_j$-regular for every $i,j\in I$.
\end{prop}

\begin{proof} This follows by \cite[Theorem~2.17]{CO1}, \cite[Proposition~3.4]{CO1} and the immediate analogues of 
\cite[16.2, 18.2]{Wis} for abelian categories.
\end{proof}

\begin{theo} \label{t:homzero-reg} 
Let $\mathcal{A}$ be an abelian category. Let $M=\bigoplus_{i\in I}M_i$ be a direct sum decomposition of 
an object $M$ of $\mathcal{A}$. Then $M$ is strongly self-regular if and only if 
$M_i$ is strongly self-regular for every $i\in I$, and ${\rm Hom}_{\mathcal{A}}(M_i,M_j)=0$ for every $i,j\in I$ with $i\neq j$.
\end{theo}

\begin{proof} This follows by \cite[Theorem~3.6]{CO1}. 
\end{proof}

Finally, we deduce a result on the structure of strongly self-regular modules over a Dedekind domain,
and in particular, on the structure of strongly self-regular abelian groups. 

\begin{coll} \label{c:dede-reg} Let $R$ be a Dedekind domain. 
\begin{enumerate}[(i)] 
\item A non-zero torsion $R$-module $M$ is strongly self-regular if and only if 
$M\cong \bigoplus_{i\in I} R/P_i$ for some distinct maximal ideals $P_i$ of $R$.
\item A non-zero finitely generated $R$-module $M$ is strongly self-regular if and only if 
$M\cong \bigoplus_{i=1}^k R/P_i$ for some distinct maximal ideals $P_i$ of $R$.
\item A non-zero injective $R$-module $M$ is strongly self-regular if and only if $M\cong K$.
\end{enumerate}
\end{coll}

\begin{proof} This follows by \cite[Corollary~3.8]{CO1}. 
\end{proof}

\begin{coll} \label{c:abgr-reg}
\begin{enumerate}[(i)] 
\item A non-zero torsion abelian group $G$ is strongly self-regular if and only if 
$G\cong \bigoplus_{i\in I} \mathbb{Z}_{p_i}$ for some distinct primes $p_i$.
\item A non-zero finitely generated abelian group $G$ is strongly self-regular if and only if 
$G\cong \bigoplus_{i=1}^k \mathbb{Z}_{p_i}$ for some distinct primes $p_i$.
\item A non-zero injective abelian group $G$ is strongly self-regular if and only if $G\cong \mathbb{Q}$.
\end{enumerate}
\end{coll}

\begin{proof} This follows by Corollary \ref{c:dede-reg}. 
\end{proof}

\begin{ex} \rm The abelian group $\mathbb{Z}_p\oplus \mathbb{Z}_p$ (for some prime $p$) is self-regular, 
but not strongly self-regular by \cite[Example~3.10]{CO1}.
\end{ex}

\section{Strongly relative regular objects: transfer via functors}

Our first result on the transfer of strongly relative regular property via (additive) functors 
involves a fully faithful covariant functor. 

\begin{theo} \label{t:ff-reg} Let $F:\mathcal{A}\to \mathcal{B}$ be an exact fully faithful covariant functor between abelian
categories, and let $M$ and $N$ be objects of $\mathcal{A}$. Then $N$ is strongly $M$-regular in $\mathcal{A}$ 
if and only if $F(N)$ is strongly $F(M)$-regular in $\mathcal{B}$.
\end{theo}

\begin{proof} This follows by \cite[Theorem~4.1]{CO1}. 
\end{proof}

In the case of strong self-regularity, let us see that one can remove the exactness of the functor from Theorem \ref{t:ff-reg}.

\begin{theo} \label{t:ff-selfreg} Let $F:\mathcal{A}\to \mathcal{B}$ be a fully faithful covariant functor between abelian
categories, and let $M$ be an object of $\mathcal{A}$. Then $M$ is strongly self-regular in $\mathcal{A}$ 
if and only if $F(M)$ is strongly self-regular in $\mathcal{B}$.
\end{theo}

\begin{proof} Assume that $M$ is strongly self-regular in $\mathcal{A}$. Let $u:F(M)\to F(M)$ be a morphism in $\mathcal{B}$. 
Since $F$ is full, we have $u=F(f)$ for some morphism $f:M\to M$ in $\mathcal{A}$. 
Since $M$ is strongly self-regular, there exists a morphism $g:M\to M$ such that $f=f^2g$. Then $u=F(f)=F(f)^2F(g)=u^2F(g)$.
This shows that $F(M)$ is strongly self-regular.

Conversely, assume that $F(M)$ is strongly self-regular in $\mathcal{B}$. Let $f:M\to M$ be a morphism in $\mathcal{A}$. 
Then there exists a morphism $v:F(M)\to F(M)$ such that $F(f)=F(f)^2v$. 
Since $F$ is full, we have $v=F(g)$ for some morphism $g:M\to M$. 
Then $F(f)=F(f)^2F(g)=F(f^2g)$, hence $f=f^2g$, because $F$ is faithful. This shows that $M$ is strongly self-regular. 
\end{proof}

For Grothendieck categories we have the following corollary.

\begin{coll} \label{c:gp-reg} Let $\mathcal{A}$ be a Grothendieck category with a generator $U$, $R={\rm
End}_{\mathcal{A}}(U)$, $S={\rm Hom}_{\mathcal{A}}(U,-):\mathcal{A}\to {\rm Mod}(R)$, and let $M$ be an object
of $\mathcal{A}$. Then $M$ is a strongly self-regular object of $\mathcal{A}$ if and only if 
$S(M)$ is a strongly self-regular right $R$-module.
\end{coll}

\begin{proof} By the Gabriel-Popescu Theorem \cite[Chapter~X, Theorem~4.1]{St}, $S$ is a fully faithful functor. 
Then the conclusion follows by Theorem \ref{t:ff-selfreg}.  
\end{proof}

It is well-known that von Neumann regularity is a Morita invariant property. 
Next we show that this is still true for strong regularity.

\begin{coll} Strong regularity is a Morita invariant property. 
\end{coll}

\begin{proof} Let $R$ and $S$ be two Morita equivalent rings via inverse equivalences $F:{\rm Mod}(R)\to {\rm Mod}(S)$ and
$G:{\rm Mod}(S)\to {\rm Mod}(R)$. We show that strong regularity of one of the rings implies strong regularity of the other. 
Assume that $S$ is a strongly regular ring. Then $S$ is a strongly self-regular right $S$-module by Proposition \ref{p:strreg}. 
If $Q=G(S)$, then it is well known that $Q$ is a progenerator in ${\rm Mod}(R)$ and $F(Q)\cong S$. Hence 
$F(Q)$ is a strongly self-regular right $S$-module, and so $Q$ is a strongly self-regular right $R$-module 
by Theorem \ref{t:ff-selfreg}. Since $Q$ is a generator in ${\rm Mod}(R)$, $R$ is isomorphic to a direct summand of 
$Q^n$ for some natural number $n\geq 1$. Since $S$ a strongly self-regular right $S$-module, 
$Q=G(S)$ is a strongly self-regular right $R$-module by Theorem \ref{t:ff-selfreg}. It follows that 
$Q^n$ is a strongly self-regular right $R$-module by Theorem \ref{t:dsreg}. Finally, 
$R$ is a a strongly self-regular right $R$-module by Corollary \ref{c:summand-reg}, 
and so $R$ is a strongly regular ring by Proposition \ref{p:strreg}.
\end{proof}

\begin{coll} \label{c:tripleff-reg} Let $(L,F,R)$ be an adjoint triple of covariant functors $F:\mathcal{A}\to \mathcal{B}$
and $L,R:\mathcal{B}\to \mathcal{A}$ between abelian categories. 
\begin{enumerate}
\item Let $M$ and $N$ be objects of $\mathcal{A}$, and assume that $F$ is fully faithful. Then $N$ is strongly 
$M$-regular in $\mathcal{A}$ if and only if $F(N)$ is strongly $F(M)$-regular in $\mathcal{B}$.
\item Let $M$ and $N$ be objects of $\mathcal{B}$, and assume that $L$ (or $R$) is fully faithful. 
Then $M$ is strongly self-regular in $\mathcal{B}$ if and only if $R(M)$ is strongly self-regular in $\mathcal{A}$
if and only if $L(M)$ is strongly self-Rickart in $\mathcal{A}$.
\end{enumerate}
\end{coll}

\begin{proof} This follows by Theorem \ref{t:ff-selfreg}.
\end{proof}

\begin{coll} Let $\varphi:R\to S$ be a ring epimorphism, and let $M$ and $N$ be right $S$-modules. Then $N$ is a 
strongly $M$-regular right $S$-module if and only if $N$ is a strongly $M$-regular right $R$-module.
\end{coll}

\begin{proof} Since $\varphi:R\to S$ is a ring epimorphism, the restriction of scalars $\varphi_*:{\rm
Mod}(S)\to {\rm Mod}(R)$ is an exact fully faithful functor \cite[Chapter~XI, Proposition~1.2]{St}. Then use Theorem \ref{t:ff-reg}.
\end{proof}

\begin{coll} Let $R$ be a strongly $G$-graded ring, and let $M$ and $N$ be right $R_e$-modules. 
Then $N$ is a strongly $M$-regular right $R_e$-module if and only if 
${\rm Ind}(N)$ is a strongly ${\rm Ind}(M)$-regular graded right $R$-module.
\end{coll}

\begin{proof} Since $R$ is a strongly $G$-graded ring, the functors ${\rm Ind},{\rm Coind}:{\rm Mod}(R_e)\to {\rm gr}(R)$ 
are functorially isomorphic. Now use \cite[Corollary~4.5]{CO1}. Alternatively, since $R$ is a strongly $G$-graded ring, 
${\rm Ind},{\rm Coind}:{\rm Mod}(R_e)\to {\rm gr}(R)$ are equivalences of categories \cite[Theorem~3.1.1]{Nasta-04} 
and use Theorem \ref{t:ff-reg}.
\end{proof}

\begin{coll} \label{c:rc-reg} Let $\mathcal{A}$ be an abelian category, $\mathcal{C}$ an abelian full subcategory of
$\mathcal{A}$ and $i:\mathcal{C}\to \mathcal{A}$ the inclusion functor. 
\begin{enumerate}
\item Assume that $\mathcal{C}$ is a coreflective subcategory of $\mathcal{A}$. 
Let $M$ and $N$ be objects of $\mathcal{C}$. Then $N$ is strongly $M$-regular in $\mathcal{C}$ if and only if 
$i(N)$ is strongly $i(M)$-regular in $\mathcal{A}$.
\item Assume that $\mathcal{C}$ is a reflective subcategory of $\mathcal{A}$. 
Let $M$ be an object of $\mathcal{C}$. Then $M$ is strongly self-regular in $\mathcal{C}$ 
if and only if $i(M)$ is strongly self-regular in $\mathcal{A}$. 
\end{enumerate}
\end{coll}

\begin{proof} (1) Note that $i$ is exact fully faithful and use Theorem \ref{t:ff-reg}.

(2) Note that $i$ is fully faithful and use Theorem \ref{t:ff-selfreg}.
 \end{proof}

\begin{coll} \label{c:com1-reg} Let $C$ be a coalgebra over a field, and  
let $M$ and $N$ be left $C$-comodules. Then $N$ is strongly $M$-regular if and only if 
$N$ is strongly $M$-regular as a right $C^*$-module. 
\end{coll}

\begin{proof} Note that ${}^C\mathcal{M}$ is a coreflective abelian subcategory of ${\rm Mod}(C^*)$
and use Corollary \ref{c:rc-reg}. 
\end{proof}

In order to discuss the transfer of the strong relative regular property to endomorphism
rings, we establish some general results involving adjoint functors.

First, we need the following property on the transfer of direct relative injectivity (projectivity).

\begin{prop} \label{p:transferdip} Let $(L,R)$ be an adjoint pair of covariant functors $L:\mathcal{A}\to \mathcal{B}$ and
$R:\mathcal{B}\to \mathcal{A}$ between abelian categories. 
\begin{enumerate}
\item Assume that $L$ is exact. Let $M$ and $N$ be objects of $\mathcal{B}$ such that $M,N\in {\rm Stat}(R)$. 
Then $M$ is direct $N$-injective in $\mathcal{B}$ if and only if $R(M)$ is direct $R(N)$-injective in $\mathcal{A}$.
\item Assume that $R$ is exact. Let $M$ and $N$ be objects of $\mathcal{A}$ such that $M,N\in {\rm Adst}(R)$. 
Then $N$ is direct $M$-projective in $\mathcal{A}$ if and only if $L(N)$ is direct $L(M)$-projective in $\mathcal{B}$.
\end{enumerate} 
\end{prop}

\begin{proof} (1) Let $\varepsilon:LR\to 1_{\mathcal{B}}$ and $\eta:1_{\mathcal{A}}\to RL$ be the counit and the unit of
adjunction respectively. 

Assume that $M$ is direct $N$-injective. Let $\alpha:P\to R(N)$ be a monomorphism with $P$ isomorphic to a 
direct summand of $R(M)$, and let $\beta:P\to R(M)$ be a morphism. Since $L$ is left exact and $M,N\in {\rm Stat}(R)$, 
$\varepsilon_NL(\alpha):L(P)\to N$ is a monomorphism, and $L(P)$ is isomorphic to a direct summand of $M$. 
By the direct $N$-injectivity of $M$, there exists a morphism $h:M\to N$ such that 
$h\varepsilon_NL(\alpha)=\varepsilon_ML(\beta)$ \cite[Lemma~5.1]{CK}. 
Let $\gamma=R(h):R(N)\to R(M)$. Since $(L,R)$ is an adjoint pair, 
we have $R(\varepsilon_N)\eta_{R(N)}=1_{R(N)}$ and $\eta$ is a natural transformation. It follows that: 
\begin{align*} \gamma\alpha&=R(h)\alpha=R(h)R(\varepsilon_N)\eta_{R(N)}\alpha=R(h)R(\varepsilon_N)RL(\alpha)\eta_P \\
& =R(h\varepsilon_NL(\alpha))\eta_P=R(\varepsilon_ML(\beta))\eta_P=
R(\varepsilon_M)RL(\beta)\eta_P=R(\varepsilon_M)\eta_{R(M)}\beta=\beta.
\end{align*}
This shows that $R(M)$ is direct $R(N)$-injective.

Conversely, assume that $R(M)$ is direct $R(N)$-injective. Let $\alpha:P\to N$ be a monomorphism with $P$ isomorphic to a 
direct summand of $M$, and let $\beta:P\to M$ be a morphism. Since $R$ is left exact, $R(\alpha):R(P)\to R(N)$ is a 
monomorphism, and $R(P)$ is clearly isomorphic to a direct summand of $R(M)$. By the direct $R(N)$-injectivity of $R(M)$, 
there exists a morphism $h:R(N)\to R(M)$ such that $hR(\alpha)=R(\beta)$  \cite[Lemma~5.1]{CK}. 
There exists a split epimorphism $p:M\to P$. Since $\varepsilon_PLR(p)=p\varepsilon_M$ is a split epimorphism, 
then so is $\varepsilon_P$. Hence there exists a morphism $r:P\to LR(P)$ such that $\varepsilon_Pr=1_P$. 
Since $(L,R)$ is an adjoint pair, $\varepsilon$ is a natural transformation. 
Also, $\varepsilon_N$ is an isomorphism, because $N\in {\rm Stat}(R)$. 
Let $\gamma=\varepsilon_ML(h)\varepsilon_N^{-1}:N\to M$. It follows that:
\begin{align*} \gamma\alpha&=\gamma\alpha\varepsilon_Pr=\varepsilon_ML(h)\varepsilon_N^{-1}\alpha\varepsilon_Pr=
\varepsilon_ML(h)\varepsilon_N^{-1}\varepsilon_NLR(\alpha)r \\ 
&=\varepsilon_ML(hR(\alpha))r=\varepsilon_MLR(\beta)r=\beta\varepsilon_Pr=\beta.
\end{align*}
This shows that $M$ is direct $N$-injective.
\end{proof}

\begin{theo} \label{t:equiv-reg} Let $(L,R)$ be an adjoint pair of covariant functors $L:\mathcal{A}\to \mathcal{B}$ and
$R:\mathcal{B}\to \mathcal{A}$ between abelian categories. 
\begin{enumerate}
\item Assume that $L$ is exact. Let $M$ and $N$ be objects of $\mathcal{B}$ such that $M,N\in {\rm Stat}(R)$. 
Then the following are equivalent:
\begin{enumerate}[(i)] 
\item $N$ is strongly $M$-regular in $\mathcal{B}$.
\item $R(N)$ is strongly $R(M)$-regular in $\mathcal{A}$ and for every morphism $f:M\to N$, ${\rm Ker}(f)$ is $M$-cyclic.
\item $R(N)$ is strongly $R(M)$-regular in $\mathcal{A}$ and for every morphism $f:M\to N$, ${\rm Ker}(f)\in {\rm Stat}(R)$.
\end{enumerate}
\item Assume that $R$ is exact. Let $M$ and $N$ be objects of $\mathcal{A}$ such that $M,N\in {\rm Adst}(R)$. 
Then the following are equivalent:
\begin{enumerate}[(i)] 
\item $N$ is strongly $M$-regular in $\mathcal{A}$.
\item $L(N)$ is strongly $L(M)$-regular in $\mathcal{B}$ and for every morphism $f:M\to N$, ${\rm Coker}(f)$ is
$N$-cocyclic.
\item $L(N)$ is strongly $L(M)$-regular in $\mathcal{B}$ and for every morphism $f:M\to N$, ${\rm Coker}(f)\in {\rm
Adst}(R)$.
\end{enumerate}
\end{enumerate}
\end{theo}

\begin{proof} This follows by \cite[Theorem~4.9]{CO1}, Theorem \ref{t:char} and Proposition \ref{p:transferdip}.
\end{proof}

Now we can extend Corollary \ref{c:gp-reg} from self-regularity to relative regularity.

\begin{coll} \label{c:gp-reg2} Let $\mathcal{A}$ be a Grothendieck category with a generator $U$, $R={\rm
End}_{\mathcal{A}}(U)$, $S={\rm Hom}_{\mathcal{A}}(U,-):\mathcal{A}\to {\rm Mod}(R)$, and let $M$ and $N$ be objects
of $\mathcal{A}$. Then $N$ is a strongly $M$-regular object of $\mathcal{A}$ if and only if 
$S(N)$ is a strongly $M$-regular right $R$-module and for every morphism $f:M\to N$, ${\rm Ker}(f)$ is $M$-cyclic.
\end{coll}

\begin{proof} By the Gabriel-Popescu Theorem \cite[Chapter~X, Theorem~4.1]{St}, $S$ is a fully faithful functor, 
hence $M\in {\rm Stat}(S)$ for every object $M$ of $\mathcal{A}$. Also, $S$ has an exact left adjoint 
$T:{\rm Mod}(R)\to \mathcal{A}$. Then the conclusion follows by Theorem \ref{t:equiv-reg}.  
\end{proof}

For contravariant functors we have the following results.

\begin{prop} \label{p:transferdip2} Let $(L,R)$ be an adjoint pair of contravariant functors $L:\mathcal{A}\to \mathcal{B}$ and
$R:\mathcal{B}\to \mathcal{A}$ between abelian categories. 
\begin{enumerate}
\item Assume that $(L,R)$ is left adjoint and $L$ is exact. 
Let $M$ and $N$ be objects of $\mathcal{B}$ such that $M,N\in {\rm Refl}(R)$. 
Then $M$ is direct $N$-injective in $\mathcal{B}$ if and only if $R(M)$ is direct $R(N)$-projective in $\mathcal{A}$.
\item Assume that $(L,R)$ is right adjoint and $R$ is exact. 
Let $M$ and $N$ be objects of $\mathcal{A}$ such that $M,N\in {\rm Refl}(L)$. 
Then $N$ is direct $M$-projective in $\mathcal{A}$ if and only if $L(N)$ is direct $L(M)$-injective in $\mathcal{B}$.
\end{enumerate} 
\end{prop}

\begin{theo} \label{t:dual-reg} Let $(L,R)$ be a pair of contravariant functors $L:\mathcal{A}\to \mathcal{B}$ and
$R:\mathcal{B}\to \mathcal{A}$ between abelian categories. 
\begin{enumerate}
\item Assume that $(L,R)$ is left adjoint and $L$ is exact. 
Let $M$ and $N$ be objects of $\mathcal{B}$ such that $M,N\in {\rm Refl}(R)$.
Then the following are equivalent:
\begin{enumerate}[(i)] 
\item $N$ is strongly $M$-regular in $\mathcal{B}$.
\item $R(M)$ is strongly $R(N)$-regular in $\mathcal{A}$ and for every morphism $f:M\to N$, ${\rm Ker}(f)$ is $M$-cyclic.
\item $R(M)$ is strongly $R(N)$-regular in $\mathcal{A}$ and for every morphism $f:M\to N$, ${\rm Ker}(f)\in {\rm Refl}(R)$.
\end{enumerate}
\item Assume that $(L,R)$ is right adjoint and $R$ is exact. 
Let $M$ and $N$ be objects of $\mathcal{A}$ such that $M,N\in {\rm Refl}(L)$. Then the following are equivalent:
\begin{enumerate}[(i)] 
\item $N$ is strongly $M$-regular in $\mathcal{A}$.
\item $L(M)$ is strongly $L(N)$-regular in $\mathcal{B}$ and for every morphism $f:M\to N$, ${\rm Coker}(f)$ is $N$-cocyclic.
\item $L(M)$ is strongly $L(N)$-regular in $\mathcal{B}$ and for every morphism $f:M\to N$, ${\rm Coker}(f)\in {\rm Refl}(L)$.
\end{enumerate}
\end{enumerate}
\end{theo}

\begin{proof} This follows by \cite[Theorem~4.10]{CO1}, Theorem \ref{t:char} and Proposition \ref{p:transferdip2}.
\end{proof}

Next we deduce the transfer of the strong self-regular property to endomorphism rings of (graded) modules. 
The following theorem generalizes \cite[Proposition~4.3]{LRR13}.

\begin{theo} \label{t:end-reg} Let $M$ be a right $R$-module, and let $S={\rm End}_R(M)$. Then:
\begin{enumerate}[(i)]
\item $M$ is a strongly self-regular right $R$-module.
\item $S$ is a strongly self-regular right $S$-module.
\item $S$ is a strongly self-regular left $S$-module.
\end{enumerate}
\end{theo}

\begin{proof} This follows by Proposition \ref{p:Mstrreg}. Alternatively, 
if $S$ is a strongly self-regular right (left) $R$-module, then $S$ is a strongly regular ring. 
It follows that $S$ is a unit regular ring, that is, for every $f\in S$, there exists an automorphism $g\in S$ 
such that $f=fgf$ \cite{Ehrlich}. By \cite[Theorem~1]{Ehrlich}, $S$ is unit regular if and only if 
$S$ is von Neumann regular and for every $f\in S$, one has ${\rm Ker}(f)\cong {\rm Coker}(f)$. Hence for every $f\in S$,  
${\rm Ker}(f)$ is $M$-cyclic and ${\rm Coker}(f)$ is $M$-cocyclic. Now conclude by \cite[Theorem~4.12]{CO1}. 
\end{proof}

\begin{coll} \label{c:endgr-reg} Let $M$ be a graded right $R$-module, and let $S={\rm END}_R(M)$. 
\begin{enumerate}[(i)]
\item $M$ is a strongly self-regular graded right $R$-module.
\item $S$ is a strongly self-regular graded right $S$-module.
\item $S$ is a strongly self-regular graded left $S$-module.
\end{enumerate}
\end{coll}

\begin{proof} This follows in a similar way as the alternative proof of Theorem \ref{t:end-reg} by using \cite[Corollary~4.13]{CO1}.
\end{proof}

\section{(Dual) strongly relative Baer objects}

In this section we begin to systematically apply our theory of (dual) strongly relative Rickart objects 
to the study of some corresponding Baer-type objects of an abelian category, called (dual) strongly relative Baer objects.

Let us first recall the following definition.

\begin{defn} \cite[Definition~6.1]{CK} \rm Let $M$ and $N$ be objects of an abelian category $\mathcal{A}$. Then $N$ is called: 
\begin{enumerate}
\item \emph{$M$-Baer} if for every family $(f_i)_{i\in I}$ with each $f_i\in \Hom_{\mathcal{A}}(M,N)$, $\bigcap_{i\in I}
{\rm Ker}(f_i)$ is a direct summand of $M$. 
\item \emph{dual $M$-Baer} if for every family $(f_i)_{i\in I}$ with each $f_i\in \Hom_{\mathcal{A}}(M,N)$, $\sum_{i\in
I} {\rm Im}(f_i)$ is a direct summand of $N$.  
\item \emph{self-Baer} if $N$ is $N$-Baer.
\item \emph{dual self-Baer} if $N$ is dual $N$-Baer.
\end{enumerate} 
\end{defn}

Now we introduce and study a particular instance of both the (dual) strongly relative Rickart property and 
the (dual) relative Baer property.

\begin{defn} \rm Let $M$ and $N$ be objects of an abelian category $\mathcal{A}$. Then $N$ is called: 
\begin{enumerate}
\item \emph{strongly $M$-Baer} if for every family $(f_i)_{i\in I}$ with each $f_i\in \Hom_{\mathcal{A}}(M,N)$, $\bigcap_{i\in I}
{\rm Ker}(f_i)$ is a fully invariant direct summand of $M$. 
\item \emph{dual strongly $M$-Baer} if for every family $(f_i)_{i\in I}$ with each $f_i\in \Hom_{\mathcal{A}}(M,N)$, $\sum_{i\in
I} {\rm Im}(f_i)$ is a fully invariant direct summand of $N$.  
\item \emph{strongly self-Baer} if $N$ is strongly $N$-Baer.
\item \emph{dual strongly self-Baer} if $N$ is dual strongly $N$-Baer.
\end{enumerate} 
\end{defn}

\begin{ex} \rm Consider the right $R$-module $M$ from Example \ref{e:strreg}. 
We have seen that $M$ is both strongly self-Rickart and dual strongly self-Rickart. 
But $M$ is neither self-Baer \cite[Example~2.18]{LRR11}, nor dual self-Baer \cite[Example~4.1]{LRR10}, 
and so it is neither strongly self-Baer, nor dual strongly self-Baer. 
\end{ex}

\begin{lemm} \label{l:Bwduo} Let $M$ and $N$ be objects of an abelian category $\mathcal{A}$. 
\begin{enumerate}
\item Assume that every direct summand of $M$ is isomorphic to a subobject of $N$. 
Then $N$ is strongly $M$-Baer if and only if $N$ is $M$-Baer and $M$ is weak duo.
\item Assume that every direct summand of $N$ is isomorphic to a factor object of $M$. 
Then $N$ is dual strongly $M$-Baer if and only if $N$ is dual $M$-Baer and $N$ is weak duo.
\end{enumerate}
\end{lemm}

\begin{proof} (1) Assume that $N$ is strongly $M$-Baer. Clearly, $N$ is $M$-Baer. Also, $N$ is strongly $M$-Rickart. 
Then $M$ is weak duo by \cite[Proposition~2.9]{CO1}. The converse is clear. 
\end{proof}

We immediately have the following corollaries.

\begin{coll} \label{c:Bwduo} Let $M$ be an object of an abelian category $\mathcal{A}$. Then:
\begin{enumerate} 
\item $M$ is strongly self-Baer if and only if $M$ is self-Baer and weak duo. 
\item $M$ is dual strongly self-Baer if and only if $M$ is dual self-Baer and weak duo.
\end{enumerate}
\end{coll}

\begin{coll} \label{c:Bindec} Let $M$ be an indecomposable object of an abelian category $\mathcal{A}$. Then:
\begin{enumerate} 
\item $M$ is strongly self-Baer if and only if $M$ is self-Baer. 
\item $M$ is dual strongly self-Baer if and only if $M$ is dual self-Baer.
\end{enumerate}
\end{coll}

\begin{prop} \label{p:Bstrring} Let $M$ be an object of an abelian category $\mathcal{A}$. 
\begin{enumerate}
\item $M$ is strongly self-Baer if and only if $M$ is self-Baer and ${\rm End}_{\mathcal{A}}(M)$ is abelian.
\item $M$ is dual strongly self-Baer if and only if $M$ is dual self-Baer and ${\rm End}_{\mathcal{A}}(M)$ is abelian.
\end{enumerate}
\end{prop}

\begin{proof} (1) If $M$ is strongly self-Baer, then $M$ is clearly self-Baer and ${\rm End}_{\mathcal{A}}(M)$ is abelian
by \cite[Proposition~2.14]{CO1}.

Conversely, assume that $M$ is self-Baer and ${\rm End}_{\mathcal{A}}(M)$ is abelian. 
Then $M$ is self-Rickart and ${\rm End}_{\mathcal{A}}(M)$ is abelian, and so 
$M$ is strongly self-Rickart by \cite[Proposition~2.14]{CO1}. Then $M$ is weak duo by \cite[Corollary~2.10]{CO1}. 
Finally, $M$ is strongly self-Baer by Corollary \ref{c:Bwduo}. 
\end{proof}

We also have the following connection between (dual) strongly relative Baer objects and (dual) strongly relative Rickart objects.

\begin{lemm} \label{l:BR} Let $M$ and $N$ be objects of an abelian category $\mathcal{A}$.
\begin{enumerate}
\item Assume that there exists the product $N^I$ for every set $I$. Then $N$ is strongly $M$-Baer if and only if $N^I$ is
strongly $M$-Rickart for every set $I$.
\item Assume that there exists the coproduct $M^{(I)}$ for every set $I$. Then $N$ is dual strongly $M$-Baer if and only
if $N$ is dual strongly $M^{(I)}$-Rickart for every set $I$. 
\end{enumerate}
\end{lemm}

\begin{proof} (1) This is immediate by using \cite[Lemma~3.11]{CO}.
\end{proof}

\begin{coll} \label{c:epimonobaer} Let $r:M\to M'$ be an epimorphism and $s:N'\to N$ a monomorphism in an abelian
category $\mathcal{A}$. 
\begin{enumerate} 
\item Assume that $\mathcal{A}$ has products. If $N$ is strongly $M$-Baer, then $N'$ is strongly $M'$-Baer.
\item Assume that $\mathcal{A}$ has coproducts. If $N$ is dual strongly $M$-Baer, then $N'$ is dual strongly $M'$-Baer.
\end{enumerate}
\end{coll}

\begin{proof} (1) If $N$ is strongly $M$-Baer, then $N^I$ is strongly $M$-Rickart for every set $I$ by Lemma~\ref{l:BR}. 
Then $N'^I$ is strongly $M'$-Rickart for every set $I$ by \cite[Theorem~2.17]{CO1}. 
Hence $N'$ is strongly $M'$-Baer by Lemma~\ref{l:BR}.
\end{proof}

\begin{coll} \label{c:dsbaer} Let $M$ and $N$ be objects of an abelian category $\mathcal{A}$, $M'$ a direct summand 
of $M$ and $N'$ a direct summand of $N$. 

(1) If $N$ is strongly $M$-Baer, then $N'$ is strongly $M'$-Baer.

(2) If $N$ is dual strongly $M$-Baer, then $N'$ is dual strongly $M'$-Baer.
\end{coll}

\begin{coll} \label{c:extensionsbaer} Let $\mathcal{A}$ be an abelian category. 
\begin{enumerate} \item Consider a short exact sequence $$0\to N_1\to N\to N_2\to 0$$ and an object $M$ in $\mathcal{A}$
such that $N_1$ and $N_2$ are strongly $M$-Baer. Then $N$ is strongly $M$-Baer.  
\item Consider a short exact sequence $$0\to M_1\to M\to M_2\to 0$$ and an object $N$ in $\mathcal{A}$ such that $N$ is
dual strongly $M_1$-Baer and dual strongly $M_2$-Baer. Then $N$ is dual strongly $M$-Baer. 
\end{enumerate}
\end{coll}

\begin{proof} (1) Since $N_1$ and $N_2$ are strongly $M$-Baer, $N_1^I$ and $N_2^I$ are strongly $M$-Rickart for every set $I$. 
We have an induced short exact sequence $0\to N_1^I\to N^I\to N_2^I\to 0$ for every set $I$. 
Then $N^I$ is strongly $M$-Rickart for every set $I$ by \cite[Theorem~2.19]{CO1}. Hence $N$ is strongly $M$-Baer by Lemma~\ref{l:BR}.
\end{proof}

(Dual) strongly relative Baer objects and (dual) strongly relative Rickart objects are also related as follows.

\begin{theo} \label{t:SS} Let $M$ and $N$ be objects of an abelian category $\mathcal{A}$.
\begin{enumerate}
\item Assume that there exists the product $N^I$ for every set $I$, and every direct summand of $M$ is isomorphic to
a subobject of $N$. Then $N$ is strongly $M$-Baer if and only if $N$ is strongly $M$-Rickart and $M$ has SSIP.
\item Assume that there exists the coproduct $M^{(I)}$ for every set $I$, and every direct summand of $N$ is isomorphic
to a factor object of $M$. Then $N$ is dual strongly $M$-Baer if and only if $N$ is dual strongly $M$-Rickart and 
$N$ has SSSP.
\end{enumerate}
\end{theo}

\begin{proof} (1) Assume that $N$ is strongly $M$-Baer. Then $N$ is strongly $M$-Rickart. 
Also, $N$ is $M$-Baer, and so $M$ has SSIP by \cite[Theorem~6.3]{CK}. 

Conversely, assume that $N$ is strongly $M$-Rickart and $M$ has SSIP. Then $N$ is $M$-Rickart and $M$ is weak duo 
by \cite[Proposition~2.9]{CO1}. It follows that $N$ is $M$-Baer by \cite[Theorem~6.3]{CK}. Finally, 
$N$ is strongly $M$-Baer by Lemma \ref{l:Bwduo}.
\end{proof}

\begin{coll} \label{c:BR} Let $M$ be an object of an abelian category $\mathcal{A}$.
\begin{enumerate}
\item Assume that there exists the product $M^I$ for every set $I$. Then $M$ is strongly self-Baer if and only if $M$ is
strongly self-Rickart and has SSIP. In particular, an indecomposable object is strongly self-Baer 
if and only if it is strongly self-Rickart.
\item Assume that there exists the coproduct $M^{(I)}$ for every set $I$. Then $M$ is dual strongly self-Baer if and only if $M$
is dual strongly self-Rickart and has SSSP. In particular, an indecomposable object is dual strongly self-Baer 
if and only if it is dual strongly self-Rickart.
\end{enumerate}
\end{coll}

Recall that a {\it (monotone) Galois connection} between two partially ordered sets $(A,\leq)$ and $(B,\leq)$ consists
of a pair $(\alpha, \beta)$ of two order-preserving functions $\alpha:A\to B$ and $\beta:B\to A$ such that for every
$a\in A$ and $b\in B$, we have $\alpha(a)\leq b \Leftrightarrow a\leq \beta(b)$ (e.g., see \cite{Erne}).

Let $M$ and $N$ be objects of an abelian category $\mathcal{A}$, and denote $U=\Hom_{\mathcal{A}}(M,N)$.
We introduce and use the following notation (see \cite{AN}).

For every subobject $X$ of $M$ and every subobject $Z$ of $U$, we denote: 
\[l_U(X)=\{f\in U\mid X\subseteq {\rm Ker}(f)\}, \quad r_M(Z)=\bigcap_{f\in Z}{\rm Ker}(f).\]

For every subobject $Y$ of $N$ and every subobject $Z$ of $U$, we denote: 
\[l'_U(Y)=\{f\in U\mid {\rm Im}(f)\subseteq Y\}, \quad r'_N(Z)=\sum_{f\in Z}{\rm Im}(f).\]

For a lattice $A$ we denote by $A^{\rm op}$ its dual lattice. One may extend \cite[Propositions~3.1, 3.2 and 3.4]{AN} to
the following theorem in abelian categories, which will be freely used.

\begin{theo} Let $M$ and $N$ be objects of an abelian category $\mathcal{A}$. Then $(r_M,l_U)$ is a Galois
connection between the subobject lattices $L(U)$ and $L(M)^{\rm op}$, and $(r'_N,l'_U)$ is a Galois connection between
the subobject lattices $L(U)$ and $L(N)$. 
\end{theo}

Recall the following concepts adapted from module theory \cite{CLVW,DHSW} (extending and lifting modules), 
\cite{Atani,WangY} (strongly extending and strongly lifting modules). 

\begin{defn} \rm Let $\mathcal{A}$ be an abelian category, $M$ an object of $\mathcal{A}$ and $K$ a subobject of
$M$.

Then $K$ is called:
\begin{enumerate}
\item an \emph{essential} subobject of $M$ if for any subobject $X$ of $M$, $K\cap X=0$ implies $X=0$. 
\item a \emph{superfluous} subobject of $M$ if for any subobject $X$ of $M$, $K+X=M$ implies $X=M$. 
\end{enumerate}

Then $M$ is called:
\begin{enumerate}
\item[(3)] \emph{(strongly) extending} if every subobject of $M$ is essential in a (fully invariant) direct summand of $M$.
\item[(4)] \emph{(strongly) lifting} if every subobject $L$ of $M$ contains a (fully invariant) direct summand $K$ of $M$ 
such that $L/K$ is superfluous in $M/K$. 
\end{enumerate}
\end{defn}

Next we introduce some relative versions of some notions which were used, under some different names, in the theory of
(dual) Baer modules \cite{KT,RR04}, and which were slightly modified in \cite{GO} to fit the theory of Baer-Galois
connections.

\begin{defn} \rm \cite[Definition~9.4]{CK} Let $M$ and $N$ be objects of an abelian category $\mathcal{A}$. 

Then $N$ is called:
\begin{enumerate}
\item \emph{$M$-$\mathcal{K}$-nonsingular} if for any morphism $f:M\to N$ in $\mathcal{A}$, ${\rm Ker}(f)$ essential in
$M$ implies $f=0$. 
\item \emph{$M$-$\mathcal{K}$-cononsingular} if for any subobject $X$ of $M$, $l_U(X)=0 $ implies that $X$ is
essential in $M$.  
\end{enumerate}

Then $M$ is called:
\begin{enumerate}
\item[(3)] \emph{$N$-$\mathcal{T}$-nonsingular} if for any morphism $f:M\to N$ in $\mathcal{A}$, ${\rm Im}(f)$ superfluous in
$N$ implies $f=0$. 
\item[(4)] \emph{$N$-$\mathcal{T}$-cononsingular} if for any subobject $Y$ of $N$, $l_U'(Y)=0$ implies that $Y$ is
superfluous in $M$.    
\end{enumerate}
\end{defn}

Now we may establish a result connecting the (dual) strongly relative Baer property to 
the strongly extending (strongly lifting) property. It is similar to \cite[Theorem~9.5]{CK} 
(also see \cite[Corollaries 3.1 and 3.2]{GO}), which relates the (dual) relative Baer property to the extending (lifting) property.

\begin{theo} \label{t:khuri} Let $M$ and $N$ be objects of an abelian category $\mathcal{A}$. 
\begin{enumerate}
\item Assume that there exists the product $N^I$ for every set $I$. Then the following are equivalent:
\begin{enumerate}[(i)]
\item $N$ is strongly $M$-Baer and $M$-$\mathcal{K}$-cononsingular.
\item $M$ is strongly extending and $N$ is $M$-$\mathcal{K}$-nonsingular.
\end{enumerate}
\item Assume that there exists the coproduct $M^{(I)}$ for every set $I$. Then the following are equivalent:
\begin{enumerate}[(i)]
\item $N$ is dual strongly $M$-Baer and $M$ is $N$-$\mathcal{T}$-cononsingular.
\item $M$ is strongly lifting and $N$-$\mathcal{T}$-nonsingular.
\end{enumerate}
\end{enumerate}
\end{theo}

\begin{proof} (1) $(i)\Rightarrow(ii)$ Assume that $N$ is strongly $M$-Baer and $M$-$\mathcal{K}$-cononsingular.
Let $L$ be a subobject of $M$. Denote $$K=r_M(l_U(L))=\bigcap\{{\rm Ker}(f)\mid f\in U \textrm{ and } L\subseteq {\rm
Ker}(f)\}.$$ Clearly, we have $L\subseteq K$. Since $N$ is strongly $M$-Baer, 
$K$ is a fully invariant direct summand of $M$, say $M=K\oplus K'$ for some subobject $K'$ of $M$. Now one shows that 
$L$ is essential in $K$ as in the proof of \cite[Theorem~9.5]{CK}. Hence $L$ is essential in 
the fully invariant direct summand $K$ of $M$, and so $M$ is strongly extending.
Finally, $N$ is $M$-$\mathcal{K}$-nonsingular also by \cite[Theorem~9.5]{CK}. 

$(ii)\Rightarrow(i)$ Assume that $M$ is strongly extending and $N$ is $M$-$\mathcal{K}$-nonsingular.
Let $I$ be a set, and let $f:M\to N^I$ be a morphism in $\mathcal{A}$. Since $M$ is strongly extending, ${\rm Ker}(f)$ is
essential in a fully invariant direct summand $L$ of $M=L\oplus L'$. Now one shows that ${\rm Ker}(f)=L$ 
as in the proof of \cite[Theorem~9.5]{CK}. Then ${\rm Ker}(f)$ is a fully invariant direct summand of $M$. 
Hence $N^I$ is strongly $M$-Rickart. It follows that $N$ is strongly $M$-Baer by Lemma~\ref{l:BR}.
Finally, $N$ is $M$-$\mathcal{K}$-cononsingular also by \cite[Theorem~9.5]{CK}.
\end{proof}

\begin{coll} Let $M$ be an object of an abelian category $\mathcal{A}$. 
\begin{enumerate}
\item Assume that there exists the product $M^I$ for every set $I$. Then the following are equivalent:
\begin{enumerate}[(i)]
\item $M$ is strongly self-Baer and $M$-$\mathcal{K}$-cononsingular.
\item $M$ is strongly extending and $M$-$\mathcal{K}$-nonsingular.
\end{enumerate}
\item Assume that there exists the coproduct $M^{(I)}$ for every set $I$. Then the following are equivalent:
\begin{enumerate}[(i)]
\item $M$ is dual strongly self-Baer and $M$-$\mathcal{T}$-cononsingular.
\item $M$ is strongly lifting and $M$-$\mathcal{T}$-nonsingular.
\end{enumerate}
\end{enumerate}
\end{coll}

\begin{proof} (1) This follows by Theorem \ref{t:khuri}. Alternatively, note that $M$ is strongly self-Baer 
if and only if $M$ is self-Baer and weak duo by Corollary \ref{c:Bwduo}, 
while $M$ is strongly extending if and only if $M$ is extending and weak duo. 
Now the result follows by \cite[Theorem~9.5]{CK}.
\end{proof}

\section{(Co)products of (dual) strongly relative Baer objects}

We continue to apply our theory of (dual) strongly relative Rickart objects in order to obtain corresponding results for 
(co)products of (dual) strongly relative Baer objects.

\begin{coll} Let $\mathcal{A}$ be an abelian category.
\begin{enumerate} 
\item Let $M$ and $N_1,\dots,N_n$ be objects of $\mathcal{A}$. Then $\bigoplus_{i=1}^n N_i$ is strongly $M$-Baer if and only if
$N_i$ is strongly $M$-Baer for every $i\in \{1,\dots,n\}$.
\item Let $M_1,\dots,M_n$ and $N$ be objects of $\mathcal{A}$. Then $N$ is dual strongly $\bigoplus_{i=1}^n M_i$-Baer if and only
if $N$ is dual strongly $M_i$-Baer for every $i\in \{1,\dots,n\}$.
\end{enumerate}
\end{coll}

\begin{proof} (1) By \cite[Theorem~3.1]{CO1} and Lemma~\ref{l:BR}, $\bigoplus_{i=1}^n N_i$ is strongly $M$-Baer if and only if 
$(\bigoplus_{i=1}^n N_i)^I\cong \bigoplus_{i=1}^n N_i^I$ is strongly $M$-Rickart for every set $I$ if and only if 
$N_1^I,\dots,N_n^I$ are strongly $M$-Rickart for every set $I$ if and only if $N_1,\dots,N_n$ are strongly $M$-Baer.
\end{proof}

\begin{lemm} \label{l:selfbaer} Let $\mathcal{A}$ be an abelian category, and let $(M_i)_{i\in I}$ be a family of
objects of $\mathcal{A}$.
\begin{enumerate} 
\item Assume that $\prod_{i\in I} M_i$ is strongly self-Baer. 
Then $M_i$ is strongly self-Baer and strongly $M_j$-Rickart for every $i,j\in I$. 
\item Assume that $\bigoplus_{i\in I} M_i$ is dual strongly self-Baer. 
Then $M_i$ is dual strongly self-Baer and dual strongly $M_j$-Rickart for every $i,j\in I$.
\end{enumerate} 
\end{lemm}

\begin{proof} (1) Assume that $\prod_{i\in I} M_i$ is strongly self-Baer. 
Then $M_i$ is strongly self-Baer and strongly $M_j$-Baer for every $i,j\in I$ by Corollary~\ref{c:epimonobaer}. 
By Theorem~\ref{t:SS}, $M_i$ is strongly $M_j$-Rickart for every $i,j\in I$.
\end{proof}

\begin{theo} Let $M_1,\dots,M_n$ be objects of an abelian category $\mathcal{A}$.
\begin{enumerate}
\item Assume that $M_i$ is direct $M_j$-injective for every $i,j\in \{1,\dots,n\}$. Then $\bigoplus_{i=1}^n M_i$ is
strongly self-Baer if and only if $M_i$ is strongly self-Baer and strongly $M_j$-Rickart for every $i,j\in \{1,\dots,n\}$.
\item Assume that $M_i$ is direct $M_j$-projective for every $i,j\in \{1,\dots,n\}$. Then
$\bigoplus_{i=1}^n M_i$ is dual strongly self-Baer if and only if $M_i$ is dual strongly self-Baer and 
dual strongly $M_j$-Rickart for every $i,j\in \{1,\dots,n\}$.
\end{enumerate}
\end{theo}

\begin{proof} (1) The direct implication follows by Lemma~\ref{l:selfbaer}.

Conversely, assume that $M_i$ is strongly self-Baer and strongly $M_j$-Rickart for every $i,j\in \{1,\dots,n\}$. 
Since $M_i$ is direct $M_j$-injective, it follows that $M_i$ is strongly $M_j$-regular for every $i,j\in \{1,\dots,n\}$ 
by Theorem~\ref{t:char}. Furthermore, $\bigoplus_{i=1}^n M_i$ is strongly self-regular by Theorem~\ref{t:dsreg}, 
and so $\bigoplus_{i=1}^n M_i$ is strongly self-Rickart by Proposition~\ref{p:Mstrreg}.
Since $M_i$ is self-Baer and $M_j$-Rickart for every $i,j\in \{1,\dots,n\}$, 
it follows that $\bigoplus_{i=1}^n M_i$ is self-Baer by \cite[Theorem~7.3]{CK}. 
Then $\bigoplus_{i=1}^n M_i$ has SSIP by \cite[Corollary~6.4]{CK}. 
Finally, $\bigoplus_{i=1}^n M_i$ is strongly self-Baer by Corollary~\ref{c:BR}.
\end{proof}

\begin{coll} Let $M_1,\dots,M_n$ be objects of an abelian category $\mathcal{A}$ such that $M_i$ is strongly $M_j$-regular for
every $i,j\in \{1,\dots,n\}$. Then $\bigoplus_{i=1}^n M_i$ is (dual) strongly self-Baer if and only if 
$M_i$ is (dual) strongly self-Baer for every $i\in \{1,\dots,n\}$.  
\end{coll}

\begin{theo} \label{t:Bhomzero} 
Let $\mathcal{A}$ be an abelian category. Let $M=\bigoplus_{i\in I}M_i$ be a direct sum decomposition of 
an object $M$ of $\mathcal{A}$. 
\begin{enumerate}
\item Then $M$ is strongly self-Baer if and only if $M_i$ is strongly self-Baer for every $i\in I$, 
${\rm Hom}_{\mathcal{A}}(M_i,M_j)=0$ for every $i,j\in I$ with $i\neq j$, and 
$N=\bigoplus_{i\in I}(N\cap M_i)$ for every direct summand $N$ of $M$.
\item Then $M$ is dual strongly self-Baer if and only if $M_i$ is dual strongly self-Baer for every $i\in I$, 
${\rm Hom}_{\mathcal{A}}(M_i,M_j)=0$ for every $i,j\in I$ with $i\neq j$, and 
$N=\bigoplus_{i\in I}(N\cap M_i)$ for every direct summand $N$ of $M$.
\end{enumerate}
\end{theo}

\begin{proof} (1) Assume first that $M$ is strongly self-Baer. Then $M$ is strongly self-Rickart, 
hence we have ${\rm Hom}_{\mathcal{A}}(M_i,M_j)=0$ for every $i,j\in I$ with $i\neq j$ by \cite[Theorem~3.6]{CO1}. 
Also, $M$ is self-Baer and weak duo by Corollary \ref{c:Bwduo}. It follows that 
$M_i$ is self-Baer for every $i\in I$ by \cite[Proposition~3.20]{RR09}. 
Also, we have $N=\bigoplus_{i\in I}(N\cap M_i)$ for every direct summand $N$ of $M$ by \cite[Theorem~2.7]{OHS}.

Conversely, assume that $M_i$ is strongly self-Baer for every $i\in I$, 
${\rm Hom}_{\mathcal{A}}(M_i,M_j)=0$ for every $i,j\in I$ with $i\neq j$, and 
$N=\bigoplus_{i\in I}(N\cap M_i)$ for every direct summand $N$ of $M$. 
Then $M_i$ is self-Baer and weak duo for every $i\in I$ by Corollary \ref{c:Bwduo}. 
It follows that $M$ is self-Baer and weak duo for every $i\in I$ by \cite[Proposition~3.20]{RR09} 
and \cite[Theorem~2.7]{OHS}. Finally, $M$ is strongly self-Baer by Corollary \ref{c:Bwduo}.

(2) This is not completely dual to (1), but it follows in a similar way as (1) by using images instead of kernels
in order to prove a property as \cite[Proposition~3.20]{RR09} for dual Baer modules.
\end{proof}

\begin{coll} \label{c:Bdede} Let $R$ be a Dedekind domain with quotient field $K$.
\begin{enumerate}
\item 
\begin{enumerate}[(i)] 
\item A non-zero torsion $R$-module $M$ is strongly self-Baer if and only if 
$M\cong \bigoplus_{i\in I} R/P_i$ for some distinct maximal ideals $P_i$ of $R$.
\item A finitely generated $R$-module $M$ is strongly self-Baer if and only if 
$M\cong J$ for some ideal $J$ of $R$ or $M\cong \bigoplus_{i=1}^kR/P_i$ 
for some distinct maximal ideals $P_i$ of $R$.
\item A non-zero injective $R$-module $M$ is strongly self-Baer if and only if $M\cong K$.
\end{enumerate}
\item An $R$-module $M$ is dual strongly self-Baer if and only if $M\cong \bigoplus_{i\in I}M_i$ 
for some distinct $R$-modules $M_i$ which are either $K$ or $E(R/P_i)$ for some maximal ideals $P_i$ of $R$, or 
$M\cong \bigoplus_{i\in I}M_i'$ for some distinct $R$-modules $M_i'$ which are either $K$ or $R/P_i$ 
for some maximal ideals $P_i$ of $R$. 
\end{enumerate}
\end{coll}

\begin{proof} (1) Since every strongly self-Baer $R$-module is strongly self-Rickart, it follows by 
\cite[Corollary~3.8]{CO1} that every self-Baer that is non-zero torsion or finitely generated or non-zero injective 
has an indecomposable direct sum decomposition. But indecomposable strongly self-Baer $R$-modules coincide with 
indecomposable strongly self-Rickart $R$-modules by Corollary \ref{c:BR}. 
Then the conclusion follows by \cite[Corollary~3.8]{CO1} and Theorem \ref{t:Bhomzero}.

(2) Every dual self-Baer $R$-module has an indecomposable direct sum decomposition \cite[Corollary~2.6]{KT}. 
The indecomposable dual self-Baer $R$-modules are $K$, $E(R/P)$ and $R/P$ for some prime ideal $P$ of $R$ 
by \cite[Theorem~3.4]{KT}. Using also the structure theorem of torsion weak duo $R$-modules, 
it follows that the indecomposable dual strongly self-Baer $R$-modules are $K$, $E(R/P)$ and $R/P$ 
for some maximal ideal $P$ of $R$. Alternatively, one can deduce the same by using Corollary \ref{c:BR} and \cite[Corollary~3.8]{CO1}.
Now the conclusion follows by \cite[Theorem~3.6]{CO1}.
\end{proof}

\begin{coll} \label{c:Babgr}
\begin{enumerate}
\item 
\begin{enumerate}[(i)] 
\item A non-zero torsion abelian group $G$ is strongly self-Baer if and only if 
$G\cong \bigoplus_{i\in I} \mathbb{Z}_{p_i}$ for some distinct primes $p_i$.
\item A finitely generated abelian group $G$ is strongly self-Baer if and only if 
$G\cong \mathbb{Z}$ or $G\cong \bigoplus_{i=1}^k \mathbb{Z}_{p_i}$ for some distinct primes $p_i$.
\item A non-zero injective abelian group $G$ is strongly self-Baer if and only if $G\cong \mathbb{Q}$.
\end{enumerate}
\item An abelian group $G$ is dual strongly self-Baer if and only if $M\cong \bigoplus_{i\in I}M_i$ 
for some distinct abelian groups $M_i$ which are either $\mathbb{Q}$ or $\mathbb{Z}_{p_i^{\infty}}$ 
for some primes $p_i$, or $M\cong \bigoplus_{i\in I}M_i'$ for some distinct abelian groups $M_i'$ 
which are either $\mathbb{Q}$ or $\mathbb{Z}_{p_i}$ for some primes $p_i$.
\end{enumerate}
\end{coll}

\begin{ex} \rm The abelian group $\mathbb{Z}\oplus \mathbb{Z}$ is self-Baer by \cite[Proposition~2.19]{RR04},
but not strongly self-Baer by Corollary \ref{c:Babgr}. 
The abelian group $\mathbb{Q}\oplus \mathbb{Q}$ is dual self-Baer by \cite[Theorem~3.4]{KT},
but not strongly dual self-Baer by Corollary \ref{c:Babgr}. 
\end{ex}

\section{(Dual) strongly relative Baer objects: transfer via functors}

Our first result on the transfer of the (dual) strongly relative Baer property
via (additive) functors involves a fully faithful covariant functor.

\begin{coll} \label{c:ffbaer} Let $F:\mathcal{A}\to \mathcal{B}$ be a fully faithful functor between abelian categories,
and let $M$ and $N$ be objects of $\mathcal{A}$. 
\begin{enumerate}
\item Assume that there exists the product $N^I$ for every set $I$, $F$ is left exact and $F$ preserves products. Then
$N$ is strongly $M$-Baer in $\mathcal{A}$ if and only if $F(N)$ is strongly $F(M)$-Baer in $\mathcal{B}$.
\item Assume that there exists the coproduct $M^{(I)}$ for every set $I$, $F$ is right exact and $F$ preserves
coproducts. Then $N$ is dual strongly $M$-Baer in $\mathcal{A}$ if and only if $F(N)$ is dual strongly $F(M)$-Baer in $\mathcal{B}$.
\end{enumerate}
\end{coll}

\begin{proof} This follows by \cite[Theorem~4.1]{CO1} and Lemma \ref{l:BR}.
\end{proof}

For Grothendieck categories we have the following corollary.

\begin{coll} Let $\mathcal{A}$ be a Grothendieck category with a generator $U$, $R={\rm End}_{\mathcal{A}}(U)$, $S={\rm
Hom}_{\mathcal{A}}(U,-):\mathcal{A}\to {\rm Mod}(R)$. Let $M$ and $N$ be objects of $\mathcal{A}$. Then $N$ is a strongly 
$M$-Baer object of $\mathcal{A}$ if and only if $S(N)$ is a strongly $S(M)$-Baer right $R$-module.
\end{coll}

\begin{proof} By the Gabriel-Popescu Theorem \cite[Chapter~X, Theorem~4.1]{St}, $S$ is a fully faithful functor which
has a left adjoint $T:{\rm Mod}(R)\to \mathcal{A}$. Since $S$ is a right adjoint, it is left exact and preserves
products. Then use Corollary \ref{c:ffbaer}. 
\end{proof}

\begin{coll} \label{c:tripleffbaer} Let $(L,F,R)$ be an adjoint triple of covariant functors $F:\mathcal{A}\to
\mathcal{B}$ and $L,R:\mathcal{B}\to \mathcal{A}$ between abelian categories. 
\begin{enumerate}
\item Let $M$ and $N$ be objects of $\mathcal{A}$, and assume that $F$ is fully faithful. 
\begin{enumerate}[(i)] 
\item Assume that there exists the product $N^I$ for every set $I$. Then $N$ is strongly $M$-Baer in $\mathcal{A}$ if and only if
$F(N)$ is strongly $F(M)$-Baer in $\mathcal{B}$.
\item Assume that there exists the coproduct $M^{(I)}$ for every set $I$. Then $N$ is dual strongly $M$-Baer in $\mathcal{A}$ if
and only if $F(N)$ is dual strongly $F(M)$-Baer in $\mathcal{B}$.
\end{enumerate}
\item Let $M$ and $N$ be objects of $\mathcal{B}$, and assume that $L$ (or $R$) is fully faithful. 
\begin{enumerate}[(i)] 
\item Assume that there exists the product $N^I$ for every set $I$. Then $N$ is strongly $M$-Baer in $\mathcal{B}$ if and only if
$R(N)$ is strongly $R(M)$-Baer in $\mathcal{A}$.
\item Assume that there exists the coproduct $M^{(I)}$ for every set $I$. Then $N$ is dual strongly $M$-Baer in $\mathcal{B}$ if
and only if $L(N)$ is dual strongly $L(M)$-Baer in $\mathcal{A}$.
\end{enumerate}
\end{enumerate}
\end{coll}

\begin{proof} This follows by \cite[Corollary~4.3]{CO1}, Lemma \ref{l:BR} and the facts that $F$ preserves products
and coproducts as a left and right adjoint, $R$ preserves products, and $L$ preserves coproducts. 
\end{proof}

\begin{coll} Let $\varphi:R\to S$ be a ring epimorphism, and let $M$ and $N$ be right $S$-modules. Then $N$ is a (dual)
strongly $M$-Baer right $S$-module if and only if $N$ is a (dual) strongly $M$-Baer right $R$-module.
\end{coll}

\begin{proof} Since $\varphi:R\to S$ is a ring epimorphism, the restriction of scalars functor $\varphi_*:{\rm
Mod}(S)\to {\rm Mod}(R)$ is fully faithful \cite[Chapter~XI, Proposition~1.2]{St}. Then use Corollary
\ref{c:tripleffbaer} for the adjoint triple of functors $(\varphi^*,\varphi_*,\varphi^!)$.
\end{proof}

\begin{coll} Let $R$ be a $G$-graded ring, and let $M$ and $N$ be right $R_e$-modules. Then: 
\begin{enumerate} 
\item $N$ is a strongly $M$-Baer right $R_e$-module if and only if 
${\rm Coind}(N)$ is a strongly ${\rm Coind}(M)$-Baer graded right $R$-module.
\item $N$ is a dual strongly $M$-Baer right $R_e$-module if and only if 
${\rm Ind}(N)$ is a dual strongly ${\rm Ind}(M)$-Baer graded right $R$-module.
\end{enumerate}
\end{coll}

\begin{proof} As in the proof of \cite[Corollary~4.5]{CO1}, the functors ${\rm Ind}$ and ${\rm Coind}$ are fully faithful. 
Now use Corollary \ref{c:tripleffbaer}. 
\end{proof}

\begin{coll} \label{c:rcbaer} Let $\mathcal{A}$ be an abelian category, $\mathcal{C}$ an abelian full subcategory of
$\mathcal{A}$ and $i:\mathcal{C}\to \mathcal{A}$ the inclusion functor. 
\begin{enumerate}
\item Assume that $\mathcal{C}$ is a reflective subcategory of $\mathcal{A}$. 
Also, assume that there exists the coproduct $M^{(I)}$ for every set $I$.
Let $M$ and $N$ be objects of $\mathcal{C}$. Then $N$ is strongly $M$-Baer in $\mathcal{C}$ 
if and only if $i(N)$ is strongly $i(M)$-Baer in $\mathcal{A}$.
\item Assume that $\mathcal{C}$ is a coreflective subcategory of $\mathcal{A}$. 
Also, assume that there exists the product $N^I$ for every set $I$.
Let $M$ and $N$ be objects of $\mathcal{C}$. Then $N$ is (dual) strongly $M$-Baer in $\mathcal{C}$ 
if and only if $i(N)$ is (dual) strongly $i(M)$-Baer in $\mathcal{A}$.
\end{enumerate}
\end{coll}

\begin{proof} This follows by Lemma \ref{l:BR} and \cite[Corollary~4.6]{CO1}.
\end{proof}

For comodule categories we have the following corollary.

\begin{coll} \label{c:com3} Let $C$ be a coalgebra over a field, and let $M$ and $N$ be left $C$-comodules. 
Then $N$ is (dual) strongly $M$-Baer if and only if $N$ is (dual) strongly $M$-Baer as a right $C^*$-module.
\end{coll}

\begin{proof} Note that ${}^C\mathcal{M}$ is a coreflective abelian subcategory of ${\rm Mod}(C^*)$
and use Corollary \ref{c:rcbaer}. 
\end{proof}

In order to discuss the transfer of the (dual) strong relative Baer property to endomorphism
rings, we give first some general results involving adjoint functors.

\begin{coll} \label{c:equivbaer} Let $(L,R)$ be an adjoint pair of covariant functors $L:\mathcal{A}\to \mathcal{B}$ and
$R:\mathcal{B}\to \mathcal{A}$ between abelian categories. 
\begin{enumerate}
\item Let $M$ and $N$ be objects of $\mathcal{B}$ such that $M,N\in {\rm Stat}(R)$ and for every set $I$ there exists
the product $N^I$. Then the following are equivalent:
\begin{enumerate}[(i)]
\item $N$ is strongly $M$-Baer in $\mathcal{B}$.
\item $R(N)$ is strongly $R(M)$-Baer in $\mathcal{A}$ and for every family $(f_i)_{i\in I}$ with each $f_i\in {\rm
Hom}_{\mathcal{A}}(M,N)$, $\bigcap_{i\in I}{\rm Ker}(f_i)$ is $M$-cyclic.
\item $R(N)$ is strongly $R(M)$-Baer in $\mathcal{A}$ and for every family $(f_i)_{i\in I}$ with each $f_i\in {\rm
Hom}_{\mathcal{A}}(M,N)$, $\bigcap_{i\in I}{\rm Ker}(f_i)\in {\rm Stat}(R)$.
\end{enumerate}
\item Let $M$ and $N$ be objects of $\mathcal{A}$ such that $M,N\in {\rm Adst}(R)$ and for every set $I$ there exists
the coproduct $M^{(I)}$. Then the following are equivalent:
\begin{enumerate}[(i)]
\item $N$ is dual strongly $M$-Baer in $\mathcal{A}$.
\item $L(N)$ is dual strongly $L(M)$-Baer in $\mathcal{B}$ and for every family $(f_i)_{i\in I}$ with each $f_i\in {\rm
Hom}_{\mathcal{A}}(M,N)$, $\sum_{i\in I} {\rm Im}(f_i)$ is $N$-cocyclic.
\item $L(N)$ is dual strongly $L(M)$-Baer in $\mathcal{B}$ and for every family $(f_i)_{i\in I}$ with each $f_i\in {\rm
Hom}_{\mathcal{A}}(M,N)$, $\sum_{i\in I} {\rm Im}(f_i)\in {\rm Adst}(R)$.
\end{enumerate}
\end{enumerate}
\end{coll}

\begin{proof} This follows by \cite[Theorem~4.9]{CO1}, Lemma \ref{l:BR}, \cite[Lemma~3.11]{CO} and the facts that $R$
preserves products and $L$ preserves coproducts.
\end{proof}

\begin{coll} Let $(L,R)$ be a pair of contravariant functors $L:\mathcal{A}\to \mathcal{B}$ and
$R:\mathcal{B}\to \mathcal{A}$ between abelian categories. 
\begin{enumerate}
\item Assume that $(L,R)$ is left adjoint. Let $M$ and $N$ be objects of $\mathcal{B}$ such that $M,N\in {\rm Refl}(R)$
and for every set $I$ there exists the product $N^I$. Then the following are equivalent:
\begin{enumerate}[(i)]
\item $N$ is strongly $M$-Baer in $\mathcal{B}$.
\item $R(M)$ is dual strongly $R(N)$-Baer in $\mathcal{A}$ and for every set $I$ and for every family $(f_i)_{i\in
I}$ with each $f_i\in {\rm Hom}_{\mathcal{A}}(M,N)$, $\bigcap_{i\in I}{\rm Ker}(f_i)$ is $M$-cyclic.
\item $R(M)$ is dual strongly $R(N)$-Baer in $\mathcal{A}$ and for every set $I$ and for every family $(f_i)_{i\in
I}$ with each $f_i\in {\rm Hom}_{\mathcal{A}}(M,N)$, $\bigcap_{i\in I}{\rm Ker}(f_i)\in {\rm Refl}(R)$.
\end{enumerate}
\item Assume that $(L,R)$ is right adjoint. Let $M$ and $N$ be objects of $\mathcal{A}$ such that $M,N\in {\rm Refl}(L)$
and for every set $I$ there exists the coproduct $M^{(I)}$. Then the following are equivalent:
\begin{enumerate}[(i)]
\item $N$ is dual strongly $M$-Baer in $\mathcal{A}$.
\item $L(M)$ is strongly $L(N)$-Baer in $\mathcal{B}$ and for every set $I$ and for every family $(f_i)_{i\in I}$ with each
$f_i\in {\rm Hom}_{\mathcal{A}}(M,N)$, $\sum_{i\in I} {\rm Im}(f_i)$ is $N$-cocyclic.
\item $L(M)$ is strongly $L(N)$-Baer in $\mathcal{B}$ and for every set $I$ and for every family $(f_i)_{i\in I}$ with each
$f_i\in {\rm Hom}_{\mathcal{A}}(M,N)$, $\sum_{i\in I} {\rm Im}(f_i)\in {\rm Refl}(L)$.
\end{enumerate}
\end{enumerate}
\end{coll}

\begin{proof} (2) This follows by \cite[Theorem~4.10]{CO1}, Lemma \ref{l:BR}, \cite[Lemma~3.11]{CO} and the fact $L$ converts
coproducts into products.
\end{proof}

Finally, we discuss the transfer of the (dual) strong relative Baer property to endomorphism rings of (graded) modules. 

\begin{coll} Let $M$ be a right $R$-module, and let $S={\rm End}_R(M)$. 
\begin{enumerate}
\item The following are equivalent:
\begin{enumerate}[(i)] 
\item $M$ is a strongly self-Baer right $R$-module. 
\item $S$ is a strongly self-Baer right $S$-module and for every set $I$ and for every family $(f_i)_{i\in I}$ with each $f_i\in
S$, $\bigcap_{i\in I}{\rm Ker}(f_i)$ is $M$-cyclic.
\item $S$ is a strongly self-Baer right $S$-module and for every set $I$ and for every family $(f_i)_{i\in I}$ with each $f_i\in
S$, $\bigcap_{i\in I}{\rm Ker}(f_i)\in {\rm Stat}({\rm Hom}_R(M,-))$.
\item $S$ is a strongly self-Baer right $S$-module and for every set $I$ and for every family $(f_i)_{i\in I}$ with each $f_i\in
S$, $\bigcap_{i\in I}{\rm ker}(f_i)$ is a locally split monomorphism.
\item $S$ is a strongly self-Baer right $S$-module and $M$ is quasi-retractable.
\end{enumerate}
\item The following are equivalent:
\begin{enumerate}[(i)] 
\item $M$ is a dual strongly self-Baer right $R$-module.
\item $S$ is a strongly self-Baer left $S$-module and for every set $I$ and for every family $(f_i)_{i\in I}$ with each
$f_i\in S$, $\sum_{i\in I} {\rm Im}(f_i)$ is $M$-cocyclic.
\item $S$ is a strongly self-Baer left $S$-module and for every set $I$ and for every family $(f_i)_{i\in I}$ with each
$f_i\in S$, $\sum_{i\in I} {\rm Im}(f_i)\in {\rm Adst}({\rm Hom}_R(M,-))$.
\item $S$ is a strongly self-Baer left $S$-module and for every set $I$ and for every family $(f_i)_{i\in I}$ with each
$f_i\in S$, $\sum_{i\in I} {\rm im}(f_i)$ is a locally split epimorphism.
\item $S$ is a strongly self-Baer left $S$-module and $M$ is quasi-coretractable.
\end{enumerate}
\end{enumerate}
\end{coll}

\begin{proof} The equivalences (i)$\Leftrightarrow$(ii)$\Leftrightarrow$(iii) follow by \cite[Theorem~4.12]{CO1}, 
Lemma \ref{l:BR} and \cite[Lemma~3.11]{CO}. 
The other equivalences follow in a similar way as the corresponding ones from \cite[Theorem~4.12]{CO1} 
with endomorphisms replaced by families of endomorphisms, kernels replaced by intersections of kernels, and
cokernels replaced by sums of images. 
\end{proof}

\begin{coll} Let $M$ be a graded right $R$-module, and let $S={\rm END}_R(M)$. 
\begin{enumerate}
\item The following are equivalent:
\begin{enumerate}[(i)] 
\item $M$ is a strongly self-Baer graded right $R$-module. 
\item $S$ is a strongly self-Baer graded right $S$-module and for every set $I$ and for every family $(f_i)_{i\in I}$ with each
$f_i\in S$, $\bigcap_{i\in I}{\rm Ker}(f_i)$ is $M$-cyclic.
\item $S$ is a strongly self-Baer graded right $S$-module and for every set $I$ and for every family $(f_i)_{i\in I}$ with each
$f_i\in S$, $\bigcap_{i\in I}{\rm Ker}(f_i)\in {\rm Stat}({\rm HOM}_R(M,-))$.
\end{enumerate}
\item The following are equivalent:
\begin{enumerate}[(i)] 
\item $M$ is a dual strongly self-Baer graded right $R$-module.
\item $S$ is a strongly self-Baer graded left $S$-module and for every set $I$ and for every family $(f_i)_{i\in I}$ with each
$f_i\in S$, $\sum_{i\in I} {\rm Im}(f_i)$ is $M$-cocyclic.
\item $S$ is a strongly self-Baer graded left $S$-module and for every set $I$ and for every family $(f_i)_{i\in I}$ with each
$f_i\in S$, $\sum_{i\in I} {\rm Im}(f_i)\in {\rm Refl}({\rm HOM}_R(-,M))$.
\end{enumerate}
\end{enumerate}
\end{coll}

\begin{proof} This follows by \cite[Corollary~4.13]{CO1}, Lemma \ref{l:BR} and \cite[Lemma~3.11]{CO}.
\end{proof}

\end{document}